\documentclass{amsart}

\usepackage{amsmath,amssymb,amsfonts}

\title[On the N{\'e}ron-Severi group of surfaces with many lines]{On the N{\'e}ron-Severi group \\ of surfaces with many lines}

\author{Samuel Boissi{\`e}re}

\address{Samuel Boissi{\`e}re, Laboratoire J.A.Dieudonn{\'e} UMR CNRS 6621,
         Universit{\'e} de Nice Sophia-Antipolis, Parc Valrose, 06108 Nice}

\email{samuel.boissiere@math.unice.fr}

\author{Alessandra Sarti}

\address{Alessandra Sarti, Johannes Gutenberg Universit{\"a}t Mainz,
Institut f{\"u}r Mathematik, 55099 Mainz, Germany}

\email{sarti@mathematik.uni-mainz.de}

\keywords{N{\'e}ron-Severi group, Picard number, lines on surfaces}

\subjclass{14J18,14J19}

\DeclareMathOperator{\NS}{\mathrm{NS}} \DeclareMathOperator{\LC}{\mathrm{LC}}
\DeclareMathOperator{\rank}{\mathrm{rk}} \DeclareMathOperator{\SL}{\mathrm{SL}} \DeclareMathOperator{\Km}{\mathrm{Km}}

\newcommand{\ii}{\mathrm{i}}

\newcommand{\IC}{\mathbb{C}}
\newcommand{\IP}{\mathbb{P}}
\newcommand{\IQ}{\mathbb{Q}}
\newcommand{\IZ}{\mathbb{Z}}

\newcommand{\cQ}{\mathcal{Q}}

\newcommand{\PC}{\mathbb{P}^3_{\mathbb{C}}}

\newtheorem{theorem}{Theorem}[section]
\newtheorem{lemma}[theorem]{Lemma}
\newtheorem{proposition}[theorem]{Proposition}

\newtheorem{remark}[theorem]{Remark}

\begin{document}

\maketitle

\begin{abstract}
For a binary quartic form $\phi$ without multiple factors, we classify the
quartic K3 surfaces $\phi(x,y)=\phi(z,t)$ whose N{\'e}ron-Severi group is
(rationally) generated by lines. For generic binary forms $\phi$, $\psi$ of prime
degree without multiple factors, we prove that the N{\'e}ron-Severi group of the
surface $\phi(x,y)=~\psi(z,t)$ is rationally generated by lines.

\end{abstract}

\section{Introduction}

The study of the N{\'e}ron-Severi group $\NS(S)$ of a given surface $S$ is
interesting for understanding its geometry, but it is not an easy task in
general. A first step is to compute its Picard number $\rho(S):=\rank\NS(S)$. A
second one is to give a family of generators of $\NS(S)$ over $\IZ$. To this
purpose, it is very useful to find first a nice family of generators of
$\NS(S)\otimes_{\IZ}\IQ$. If one already knows the value of the determinant of
$\NS(S)$, this can help deducing a family of generators. If not, the study of
the \emph{rational} generators gives non trivial information for the value of
the discriminant.

Let $\phi$ be a binary quartic form without multiple factors. After a suitable
linear change of coordinates, we may assume that $\phi$ is of the form:
$$
\phi(x,y)=yx(y-x)(y-\lambda x)
$$
for $\lambda\in\IC\setminus\{0,1\}$.
Naturally associated to $\phi$ are the K3 surface
$S_\phi:\phi(x,y)=\phi(z,t)$ and the elliptic curve $E_\phi:t^2=\phi(1,y)$.

\begin{remark}\label{remark:iso}
Observe that if $\phi,\phi'$ are the forms corresponding to $\lambda,\lambda'$ and $\lambda'$ is one of the values $\lambda,\frac{1}{\lambda},1-\lambda,\frac{1}{1-\lambda},\frac{\lambda}{\lambda-1},\frac{\lambda-1}{\lambda}$ then there is a linear isomorphism $S_\phi\cong S_{\phi'}$.
\end{remark}

The interplay between the geometry of the K3 surface $S_\phi$ and the arithmetic of
the elliptic curve $E_\phi$ has been studied by many authors. Of particular
interest is the link between the value of the Picard number $\rho(S_\phi)$ and
the existence of a complex multiplication on $E_\phi$. The following result is
classical (see \cite{Kuwata} and references therein):
$$
\rho(S_\phi)=\begin{cases} 20 & \text{if } E_\phi \text{ has a complex
multiplication,}\\ 19 &\text{otherwise}.\end{cases}
$$

We pursue the study by giving numerical conditions for the N{\'e}ron-Severi group
of $S_\phi$ to be \emph{rationally generated by lines}:

\noindent{\bf Notation -- Definition.} Let $S\subset\PC$ be a smooth surface of degree $d\geq 3$. If $L$ is a line contained
in $S$, by the genus formula the self-intersection of $L$ in $S$ is $L^2=-d+2$, so
the class of $L$ in $\NS(S)$ is not a torsion class. We denote by $\LC(S)$ the
sublattice of the torsion-free part of $\NS(S)$ generated by the classes of the
lines contained in $S$. For a generic surface $S$, it is well-known that
$\LC(S)=0$. If not, these classes are natural candidates as generators of
$\NS(S)$ and we say that $\NS(S)$ is \emph{rationally generated by lines} if
$\rank\LC(S)=\rho(S)$, that is $\LC(S)\otimes_{\IZ}\IQ=\NS(S)\otimes_{\IZ}\IQ$.

The most famous examples of surfaces whose N{\'e}ron-Severi group is rationally
generated by lines are certain Fermat surfaces (see \cite{ShiodaENS}). The
surfaces we study here are a natural generalization of them. We prove (\S\ref{proof:th:K3}):

\begin{theorem}\label{th:K3} The N{\'e}ron-Severi group of $S_\phi$ is rationally generated by
lines exactly in the following cases:
\begin{enumerate}
\item $\lambda\notin\overline{\IQ}$;

\item
$\lambda\in\{-1,2,\frac{1}{2},\frac{1+\ii\sqrt{3}}{2},\frac{1-\ii\sqrt{3}}{2}\}$;

\item
$\lambda\in\overline{\IQ}\setminus\{-1,2,\frac{1}{2},\frac{1+\ii\sqrt{3}}{2},\frac{1-\ii\sqrt{3}}{2}\}$
and $\rho(S_\phi)=19$.
\end{enumerate}
\end{theorem}

Looking now for a set of generators of the N\'eron-Severi group, we prove (\S\ref{s:discr}):

\begin{theorem}\label{th:overZ}
The N{\'e}ron-Severi group of $S_\phi$ is  generated by
lines only in case~$(2)$.
\end{theorem}

Generalizing the construction, one can consider two binary forms $\phi,\psi$ of
degree $d$ without multiple factors and the associated surface
$S^d_{\phi,\psi}:\phi(x,y)=\psi(z,t)$. One can prove that
$\rho(S^d_{\phi,\psi})\geq (d-1)^2+1$ with equality for $d$ prime and
$\phi,\psi$ generic (see \cite{Sasakura}). We prove (\S\ref{s:prime}):

\begin{theorem}\label{th:prime}
For $d$ prime and $\phi,\psi$ generic, the N{\'e}ron-Severi group of
$S^d_{\phi,\psi}$ is rationally generated by lines.
\end{theorem}

In Theorem \ref{th:K3} we do not consider the quartics $S^4_{\phi,\psi}$ for
$\phi\neq\psi$ since, although $\rho(S^4_{\phi,\psi})=18$ (see again \cite{Kuwata}), Proposition \ref{prop:LCprime} below says that their 16 lines generate an
intersection matrix of rank 10, so such surfaces do not enter in our context.

{\it We thank the referee for helpful suggestions and comments.}

\section{Proof of Theorem \ref{th:K3}}
\label{proof:th:K3}

The result follows from the following proposition:
\begin{proposition}\label{prop:rankK3} If $\lambda
\in\{-1,2,\frac{1}{2},\frac{1+\ii\sqrt{3}}{2},\frac{1-\ii\sqrt{3}}{2}\}$, then
$\rank \LC(S_\phi)=20$, otherwise $\rank \LC(S_\phi)=19$.
\end{proposition}

\begin{proof}[Proof of Theorem \ref{th:K3}] Assuming Proposition
\ref{prop:rankK3}, we prove Theorem \ref{th:K3}. The key argument is that if
$E_\phi$ has a complex multiplication, then its $j$-invariant is algebraic over
$\overline{\IQ}$ (see \cite{Silverman}). Since
$j(E_\phi)=\frac{256(1-\lambda+\lambda^2)^3}{\lambda^2(\lambda-1)^2}$,
$j(E_\phi)\in\overline{\IQ}$ if and only if $\lambda\in \overline{\IQ}$. Then:

\noindent - If $\lambda\notin\overline{\IQ}$, $E_\phi$ has no complex
multiplication so $\rho(S_\phi)=19$ and by Proposition \ref{prop:rankK3},
$\rank\LC(S_\phi)=19$. This proves (1).

\noindent - If
$\lambda\in\{-1,2,\frac{1}{2},\frac{1+\ii\sqrt{3}}{2},\frac{1-\ii\sqrt{3}}{2}\}$,
by Proposition \ref{prop:rankK3} we have $\rank\LC(S_\phi)=20$ so
$\rho(S_\phi)=20$. This proves (2).

\noindent - If
$\lambda\in\overline{\IQ}\setminus\{-1,2,\frac{1}{2},\frac{1+\ii\sqrt{3}}{2},\frac{1-\ii\sqrt{3}}{2}\}$,
then $\rho(S_\phi)\in\{19,20\}$ and $\rank \LC(S_\phi)=19$. This gives (3).
\end{proof}

\begin{remark}
In case (3) of Theorem \ref{th:K3}, one can not be more precise since:
\begin{itemize}
\item When $j(E_\phi)\in\overline{\IQ}$ (so $\lambda\in\overline{\IQ}$), it is
not clear whether $E_\phi$ admits a complex multiplication or not.

\item There is a dense and numerable set of $\lambda\in\overline{\IQ}$ such
that $\rho(S_\phi)=20$ (see \cite{Oguiso}).
\end{itemize}
\end{remark}

\begin{proof}[Proof of Proposition \ref{prop:rankK3}] The description of the
lines on $S_\phi$ comes from Segre \cite{Segre}. We follow the presentation
given in \cite{BoissiereSarti}.

\medskip

\paragraph{\it Case 1.} If $\lambda\notin
\{-1,2,\frac{1}{2},\frac{1+\ii\sqrt{3}}{2},\frac{1-\ii\sqrt{3}}{2}\}$, the
group of automorphisms of $\IP^1_\IC$ permuting the set
$\{\infty,0,1,\lambda\}$ is the dihedral group $D_2=\{id,s_1,s_2,s_1s_2\}$ and
the surface $S_\phi$ contains exactly the following 32 lines:
\begin{align*}
\ell_z(u,v)\colon&\begin{cases} vx=uy\\vt=uz\end{cases} &
\ell_{id}(p)\colon&\begin{cases} x=pz\\y=pt\end{cases} &
\ell_{s_1}(p)\colon&\begin{cases} x=pz-pt\\y=\lambda
pz-pt\end{cases}\\
&\scriptstyle u,v\in\{\infty,0,1,\lambda\} & &\scriptstyle
p\in\{1,-1,\ii,-\ii\}&&\scriptstyle
p\in\{\frac{1}{\sqrt{\lambda-1}},\frac{-1}{\sqrt{\lambda-1}},\frac{\ii}{\sqrt{\lambda-1}},\frac{-\ii}{\sqrt{\lambda-1}}\}\\
\end{align*}
\begin{align*}
\ell_{s_2}(p)\colon&\begin{cases} x=pt\\y=\lambda pz\end{cases} &
\ell_{s_1s_2}(p)\colon&\begin{cases} x=-\lambda pz+pt\\y=-\lambda pz+\lambda pt\end{cases}\\
&\scriptstyle
p\in\{\frac{1}{\sqrt{\lambda}},\frac{-1}{\sqrt{\lambda}},\frac{\ii}{\sqrt{\lambda}},\frac{-\ii}{\sqrt{\lambda}}\}
&&\scriptstyle
p\in\{\frac{1}{\sqrt{\lambda^2-\lambda}},\frac{-1}{\sqrt{\lambda^2-\lambda}},\frac{\ii}{\sqrt{\lambda^2-\lambda}},\frac{-\ii}{\sqrt{\lambda^2-\lambda}}\}
\end{align*}

The intersection matrix of these 32 lines is easy to compute (we do not
reproduce it here), and is independent of $\lambda$. One finds that its rank is
$19$, so $\rank \LC(S_\phi)=19$.

\medskip

\paragraph{\it Case 2.} If $\lambda\in\{-1,2,\frac{1}{2}\}$, the surfaces are isomorphic to each other by Remark \ref{remark:iso}. The group
of automorphisms is the dihedral group $D_4=\langle D_2,r\rangle$. The surface
$S_\phi$ contains exactly 48 lines: the 32 preceding ones and  16 other lines. For $\lambda=-1$ for example, these lines are:
\begin{align*}
\ell_r(p)\colon&\begin{cases} x=pz+pt\\y=-pz+pt\end{cases} &
\ell_{r^{-1}}(p)\colon&\begin{cases} x=-pz+pt\\y=-pz-pt\end{cases}&\scriptstyle
p\in\{\frac{1+\ii}{2},\frac{1-\ii}{2},\frac{-1+\ii}{2},\frac{-1-\ii}{2}\}\\
\ell_{rs_1}(p)\colon&\begin{cases} x=pt\\y=pz\end{cases} &
\ell_{s_1r}(p)\colon&\begin{cases} x=-pz\\y=pt\end{cases} & \scriptstyle
p\in\{\frac{1+\ii}{\sqrt{2}},\frac{1-\ii}{\sqrt{2}},\frac{-1+\ii}{\sqrt{2}},\frac{-1-\ii}{\sqrt{2}}\}
\end{align*}
The rank of the intersection matrix of the 48 lines is $\rank
\LC(S_\phi)=20$.

\medskip

\paragraph{\it Case 3.} If
$\lambda\in\{\frac{1+\ii\sqrt{3}}{2},\frac{1-\ii\sqrt{3}}{2}\}$, the surfaces are isomorphic to each other by Remark \ref{remark:iso}. The group of
automorphisms is the tetrahedral group $T=~\langle r,s\rangle$. The surface
$S_\phi$ contains exactly the following 64 lines:
\begin{align*}
\ell_{z}(u,v)\colon&\begin{cases}vx=uy\\vt=uz\end{cases} &
\ell_{id}(p)\colon&\begin{cases} x=pz\\y=pt
\end{cases} & \substack{u,v\in\{\infty,0,1,\lambda\}\\p\in\{1,-1,\ii,-\ii\}}\\
\ell_{r}(p)\colon&\begin{cases} x=pz\\y=pz+\lambda^2 pt \end{cases} & \ell_{r^2}(p)&:\begin{cases} x=pz\\
y=\lambda pz-\lambda pt\end{cases} \\
\end{align*}
\begin{align*}
\ell_{s}(p)\colon&\begin{cases} x=pt\\y=\lambda pz \end{cases} &
\ell_{rs}(p)\colon&\begin{cases} x=pt\\y=-pz+pt
\end{cases} & \ell_{rsr}(p)&:\begin{cases} x=pz+\lambda^2 pt\\y=\lambda^2pt \end{cases} \\
\ell_{r^2s}(p)\colon&\begin{cases} x=pt\\y=-\lambda^2 pz+\lambda pt \end{cases} &
\ell_{sr}(p)\colon&\begin{cases} x=pz+\lambda^2 pt\\y=\lambda pz
\end{cases} & \scriptstyle p&\scriptstyle\in\{\lambda,-\lambda,\ii\lambda,-\ii\lambda\}
\end{align*}
\begin{align*}
\ell_{rsr^2s}(p)\colon&\begin{cases} x=-\lambda^2pz+\lambda
pt\\y=-\lambda^2pz+\lambda^2pt \end{cases} & \ell_{r^2srs}(p)\colon&\begin{cases}
x=-pz+pt\\y=-\lambda pz+pt
\end{cases} & \scriptstyle
p\in\{\lambda^2,-\lambda^2,\ii\lambda^2,-\ii\lambda^2\} \\
\ell_{srs}(p)\colon&\begin{cases}x=-pz+pt \\y=\lambda pt \end{cases} &
\ell_{rsrs}(p)\colon&\begin{cases} x=-pz+pt\\y=-pz \end{cases}
\end{align*}
The rank of the intersection matrix of the 64 lines is $\rank
\LC(S_\phi)=20$.

\end{proof}

\section{Proof of Theorem \ref{th:overZ}}
\label{s:discr}

As we explained in the Introduction, once one has found a nice family of
\emph{rational} generators of the N{\'e}ron-Severi group, the next task is to
get information on divisible classes. We call a divisor $\Lambda=\sum_{i=1}^{n} \alpha_i L_i$$\in \NS(S)$ {\it $2^m$-divisible} if the class of $\Lambda$ in $\NS(S)$ is divisible by $2^m$; for $m=1$ we say also that the lines in $\Lambda$ form an {\it even set}.

\begin{proof}[Proof of Theorem \ref{th:overZ}]\text{}
\paragraph{\it Cases $(1)$ and $(3)$.} For
$\lambda\notin\{-1,2,\frac{1}{2},\frac{1+\ii\sqrt{3}}{2},\frac{1-\ii\sqrt{3}}{2}\}$,
with the help of a computer program we obtain that the best choice of a family of 19 lines among the 32 generating
rationally the N{\'e}ron-Severi group gives a determinant of value $2^9$. Denoting this lattice by $M$ and its dual by $M^\vee$, the discriminant group is:
$$
M^{\vee}/M=(\IZ_2)^{\oplus 2}\oplus (\IZ_4)^{\oplus 2}\oplus \IZ_8
$$
hence we can have only $2^m$-divisible classes for $m=1,2,3$. Denote by $(M^{\vee}/M)_2$ the part of the discriminant group generated by the $2$-torsion classes. We have $(M^{\vee}/M)_2=(\IZ_2)^{\oplus 5}$ hence rank$(M^{\vee}/M)_2=5$. However, denoting by $T$ the transcendental lattice of $S_\phi$, $(\NS(S_\phi)^\vee/\NS(S_\phi))_2\cong (T^\vee/T)_2$  has rank at most the rank of $T$, which is three: This shows that $M\subsetneq \NS(S_\phi)$, and that there are at least two even sets of lines in the N{\'e}ron Severi group. In particular there is no set of $19$ lines generating $ \NS(S_\phi)$.

\medskip

\paragraph{\it Case $(2)$ for $\lambda\in\{-1,2,\frac{1}{2}\}$.} By Remark \ref{remark:iso}, the surfaces $S_\phi$ are isomorphic to each other. The best choice of a family of 20
lines among 48 gives a determinant of value $-2^6$. Observe that a suitable permutation of the zeros of $x^4-y^4$ in $\IP^1_\IC$ gives a cross-ratio equal to $-1$, so our surfaces are isomorphic to the Fermat quartic. It is then well-known that $\det\NS(S_\phi)=-64$, so the lines generate the N{\'e}ron-Severi group.

\medskip

\paragraph{\it Case $(2$) for $\lambda\in\{\frac{1+\ii\sqrt{3}}{2},\frac{1-\ii\sqrt{3}}{2}\}$.} A computer program shows that the best choice of a family of 20 lines among
the 64 contained in the surface, generating rationally the N{\'e}ron-Severi group,
gives a determinant of value $-2^4\cdot 3$. We show in Appendix \ref{s:Kummer} that $\det\NS(S_\phi)=-48$ so the lines generate the N\'eron-Severi group.
\end{proof}

\section{Proof of Theorem \ref{th:prime}}
\label{s:prime}

Since $\rho(S^d_{\phi,\psi})=(d-1)^2+1$ for $d$ prime and $\phi,\psi$ generic,
Theorem \ref{th:prime} follows from the following result:

\begin{proposition}\label{prop:LCprime} It is $\rank \LC(S^d_{\phi,\psi})=(d-1)^2+1$.
\end{proposition}

\begin{proof}[Proof of Proposition \ref{prop:LCprime}] We set $S:=S^d_{\phi,\psi}$.
Let $L$ be the line $z=t=0$ and $L'$ be the line $x=y=0$. The intersection
$S\cap L$ is the set of zeros of $\phi$, whereas $S\cap L'$ is the set of zeros
of $\psi$. If $p\in L$ is a zero of $\phi$ and $q\in L'$ a zero of $\psi$, the
line $L_{p,q}$ joining $p$ and $q$ is contained in $S$: this gives a family of
$d^2$ lines contained in $S$. The intersection matrix of this family is given
by $L^2=-d+2$ and $L\cdot L'=1$ if $L$ and $L'$ intersect, $0$ otherwise. Note
that:
$$
(L_{p,q}\cap L_{p',q'}\neq \emptyset) \Longleftrightarrow (p=p' \text{ or }
q=q').
$$
This implies that after ordering correctly the lines, the intersection matrix
is the matrix $M_d:=K^d_{-d+2,1,1,0}$ (see the notation in Appendix
\ref{s:appmat}). Remark \ref{rem:invers} gives $\rank \LC(S)=\rank
M_d=(d-1)^2+1$.
\end{proof}

\appendix

\section{Some linear algebra}\label{s:appmat}

Let $a,b,c,d,\ldots$ denote indeterminates. For $d\geq 2$, let $J^d_{a,b}$ be
the $(d,d)$-matrix defined by:
$$
J^d_{a,b}:=\left(\begin{matrix}a & & \text{\Huge{$b$}}\\ & \ddots & \\
\text{\Huge{$b$}} & &a\end{matrix}\right) =b\cdot
\left(\text{\Huge{$1$}}\right)+(a-b)\cdot I_d
$$
where $I_d$ denotes the identity $(d,d)$-matrix. The following lemma is clear:
\begin{lemma}\label{lem:Jab} The following identities hold:
\begin{align*}
J^d_{a,b}+J^d_{a',b'}&=J^d_{a+a',b+b'};\\
J^d_{a,b}\cdot J^d_{a',b'}&=J^d_{aa'+(d-1)bb',ab'+a'b+(d-2)bb'}.
\end{align*}
\end{lemma}

Let now $K^d_{a,b,c,d}$ be the $(d^2,d^2)$-matrix defined as the following
$(d,d)$-blocks of $(d,d)$-matrices:
$$
K^d_{a,b,c,d}:=\left(\begin{matrix}J^d_{a,b} & & \text{\Huge{$J^d_{c,d}$}}\\ & \ddots & \\
\text{\Huge{$J^d_{c,d}$}} & &J^d_{a,b}\end{matrix}\right)
$$

The following lemma follows easily from Lemma \ref{lem:Jab}:
\begin{lemma}\label{lem:Mabcd} The following identity holds:
$$
K^d_{a,b,c,d}\cdot K^d_{a',b',c',d'}=K^d_{\alpha,\beta,\gamma,\delta}
$$
where:
\begin{align*}
\alpha&=aa'+(d-1)(bb'+cc')+(d-1)^2dd'; \\
\beta&= ab'+a'b+(d-1)(cd'+c'd)+(d-2)bb'+(d-1)(d-2)dd';\\
\gamma&=ac'+a'c+(d-1)(bd'+b'd)+(d-2)cc'+(d-1)(d-2)dd';\\
\delta&=ad'+a'd+bc'+b'c+(d-2)(cd'+c'd+bd'+b'd)+(d-2)^2dd'.
\end{align*}
\end{lemma}

Set $K_d:=K^d_{1,1,1,0}$. Its minimal polynomial $\mu_{K_d}(t)$ is given by:
\begin{lemma}\label{lem:muKd}
$\mu_{K_d}(t)=(t-(d-1))\cdot(t-(2d-1))\cdot(t+1)$.
\end{lemma}

\begin{proof}
Note that:
\begin{align*}
K_d-(d-1)I_d&=K^d_{-d+2,1,1,0};\\
K_d-(2d-1)I_d&=K^d_{-2d+2,1,1,0};\\
K_d+I_d&=K^d_{2,1,1,0}.
\end{align*}
Applying Lemma \ref{lem:Mabcd} one gets:
\begin{align*}
K^d_{-d+2,1,1,0}\cdot K^d_{-2d+2,1,1,0}&=K^d_{2(d-1)^2,-2d+2,-2d+2,2};\\
K^d_{-d+2,1,1,0}\cdot K^d_{2,1,1,0}&=K^d_{2,2,2,2};\\
K^d_{-2d+2,1,1,0}\cdot K^d_{2,1,1,0}&=K^d_{-2d+2,-d+2,-d+2,2};\\
K^d_{-d+2,1,1,0}\cdot K^d_{-2d+2,1,1,0}\cdot K^d_{2,1,1,0}&=K^d_{0,0,0,0}=0.
\end{align*}
\end{proof}

For $\lambda\in \{d-1,2d-1,-1\}$, we denote by $V(\lambda)$ the eigenspace of
$K_d$ associated to the eigenvalue $\lambda$. One computes:
\begin{lemma}\label{lem:dim}
$$
\dim V(2d-1)=1;\quad \dim V(-1)=(d-1)^2;\quad \dim V(d-1)=2(d-1).
$$
\end{lemma}

\begin{proof}
The first two results are a (quite long) direct computation. One deduces the
third one using that $K_d$ is diagonalizable (Lemma \ref{lem:muKd}).
\end{proof}

\begin{remark}\label{rem:invers}
Since $K^d_{\lambda,1,1,0}=K_d-(1-\lambda)I_d$, the matrix
$K^d_{\lambda,1,1,0}$ is invertible when $1-\lambda$ is not an eigenvalue of
$K_d$. By Lemma \ref{lem:muKd} this is $\lambda\notin\{-d+2,-2d+2,2\}$. For
$\lambda=-d+2$, one has:
$$
\rank K^d_{-d+2,1,1,0}=d^2-\dim V(d-1)=(d-1)^2+1.
$$
\end{remark}

\section{Results on Kummer surfaces}\label{s:Kummer}

We recall some classical facts from \cite{Inose,ShapiroShafarevich,InoseShioda, ShiodaMitani}. If $S$ is a K3 surface with Picard number $20$, we denote by $T_S$ the transcendental lattice and $Q_S$ the intersection matrix of $T_S$ with respect to an oriented basis. Let $\cQ$ be the set of positive definite, even integral $2\times 2$ matrices. The class $[Q_S]\in\cQ/\SL_2(\IZ)$ is uniquely determined by $S$ and $\det\NS(S)=-\det Q_S$.

For $S_\phi$, let $\sigma$ be the involution $(x:y:z:t)\mapsto(x:y:-z:-t)$. Then the minimal resolution of $S_\phi/\sigma$ is isomorphic to the Kummer surface $Y:=\Km(E_\phi\times E_\phi)$ and:
$$
Q_{S_\phi}=2Q_Y=4Q_A
$$
where $A:=E_\phi\times E_\phi$ and $Q_A$ is the binary quadratic form associated to $A$ as in \cite{ShiodaMitani}.

For $\lambda=\frac{1+\ii\sqrt{3}}{2}$, the group of automorphisms of the elliptic curve $E_\phi$ fixing a point has order $6$ (since $j(\lambda)=0$) so $E_\phi\cong C_\tau:=\IC/\IZ+\tau\IZ$ with $\tau=\frac{-1+\ii\sqrt{3}}{2}$. By the construction of \cite{ShiodaMitani}, for $A=C_\tau\times C_\tau$, one has $Q_A=\left(\begin{matrix} 2 & 1\\1 & 2\end{matrix}\right)$ so $Q_{S_\phi}=\left(\begin{matrix} 8 & 4\\4& 8\end{matrix}\right)$ and $\det\NS(S_\phi)=-\det Q_{S_\phi}=-48$. Moreover, observe that for $A'=C_\tau\times C_{\tau'}$ with $\tau'=\ii\sqrt{3}$, one has $Q_{A'}=\left(\begin{matrix} 4 & 2\\2& 4\end{matrix}\right)$ so $S_\phi\cong\Km(A')$.

\begin{remark}
The same method has been used to compute the determinant of the N\'eron-Severi group of the Fermat quartic.
\end{remark}

\bibliographystyle{amsalpha}
\bibliography{BoissiereSartiPicardBib}
\end{document}